\documentclass{article}
\usepackage{graphicx}
\usepackage{hyperref} 
\pdfoutput=1
\hypersetup{colorlinks=true, linkcolor=blue}
\usepackage[justification=raggedright,singlelinecheck=false]{caption}
\date{}

\usepackage{amsmath, amssymb, amsthm}
\usepackage{cleveref}
\crefname{proposition}{Proposition}{Proposition}
\usepackage{bm}
\usepackage{booktabs}
\usepackage{array}
\usepackage{geometry}

\usepackage{tabularx}
\usepackage{booktabs} 

\usepackage[utf8]{inputenc}
\usepackage[T1]{fontenc}
\usepackage{hyperref}
\hypersetup{
    colorlinks=true,
    linkcolor=blue,
    urlcolor=blue,
    citecolor=blue
}

\usepackage{hyperref}
\usepackage{float}
\usepackage{bbm}
\usepackage{mathrsfs}
\usepackage{tikz}
\usepackage{setspace}
\title{PARAMETRIX CONSTRUCTION AND INVERTIBILITY OF $\Delta_g + 1$}
\author{SHIH-YU CHANG}

\begin{document}

\maketitle
\begin{abstract}
Given a manifold $M$ with an asymptotically cylindrical end, we construct the parametrix for the shifted Laplace operator $\Delta_g + 1$ on the blown-up $b$-double space $M_b^2$. We will review the geometry of manifold with a cylindrical end and construct the inverse directly via separation of variables, which motivates the parametrix construction. Next, we review the mapping properties and show how to get elliptic regularity from the mapping properties of the parametrix. 

Through this explicit example, we hope this paper provides another reference for the theory of $b$-operators introduced by Melrose \cite{Mel93}.
\end{abstract}

\tableofcontents
\newpage                     

\section{Introduction}

The analysis of differential operators on noncompact manifolds requires specialized machinery to handle the lack of compactness at infinity. An example of this will be manifold with a cylindrical end, a geometry that appears naturally across geometric analysis. Melrose constructed the $b$-calculus to deal with such problems. For a comprehensive exposition of $b$-calculus and its foundational geometric applications, the reader is referred to the textbooks by Melrose \cite{Mel93, Mel96, Mel03} and notes by Grieser \cite{Gri01}. Additionally, for a treatment of pseudodifferential operators on manifolds, including detailed propositions on their boundedness on Sobolev spaces, Chapter~5 of Hintz's book \cite{Hin25} is also an excellent resource. 

Dyatlov and Li \cite{DyLi24} provided an introduction to the $b$-calculus framework using $\Delta_g$ on manifold with cylindrical ends as a motivating example. They cut out the zero to achieve invertibility. Here, the $+1$ shift introduces a strict positive spectral gap, which naturally eliminates zero mode obstructions.

In this paper, we focus on the special case $\Delta_g + 1$ to show how $b$-calculus is applied to tackle invertibility. In Section 1, we introduce the geometry of manifold with a cylindrical end. In Section 2, we discuss the positive definite method. We solve the problem using separation of variables in Section 3, but we cannot do separation of variables globally because of the compact core.
Our manifold is only asymptotically cylindrical ($\mathbb{R} \times B$) due to a 
compact core. Thus, the metric has a 
remainder $\mathcal{O}$:
\[
\Delta_g + 1 = \Delta_{B \times \mathbb{R}} + \mathcal{O}
\]
This restricts the separation of variables to a local level. To complete the global 
construction, we handle this remainder using parametrix construction in Section 4.

In Section 4, we construct the Schwartz kernel of the parametrix on the double space $M_b^2$. This requires a blow-up, the computational details in Section 3 illustrate why this blow-up is necessary. We also analyze the behavior of the parametrix in Section 4.

For Section 5, we will see the mapping properties of the parametrix, which will give us the elliptic regularity.

\subsection{Manifold with a Cylindrical End}

\subsubsection{Manifold Decomposition}
We start with a manifold $M$ with a cylindrical end. We then compactify the manifold by adding an infinity, making it a compact manifold with boundary. We are assuming our metric is exponentially close to the product metric.

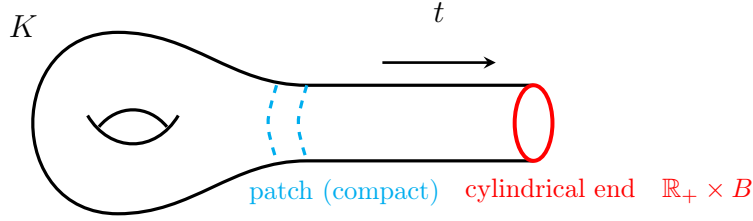
\usetikzlibrary{calc}
\begin{center} 
\begin{tikzpicture}[line width=1.2pt]

    \draw (4, 0.5) -- (1, 0.5) 
        .. controls (0, 0.5) and (-0.5, 1.2) .. (-1.5, 1.2)
        .. controls (-3, 1.2) and (-3, -1.2) .. (-1.5, -1.2)
        .. controls (-0.5, -1.2) and (0, -0.5) .. (1, -0.5)
        -- (4, -0.5);

   \draw[cyan, dashed, line width=1.2pt] 
    (0.6, 0.5) .. controls (0.45, 0) .. (0.6, -0.5); 
  
   \draw[cyan, dashed, line width=1.2pt] 
    (1.0, 0.5) .. controls (0.85, 0) .. (1.0, -0.5);

  \node[cyan, anchor=north] at (1.5, -0.6) {patch (compact)};
  \node[red, anchor=north] at (5, -0.6) {cylindrical end\quad$\mathbb{R}_+ \times B$};

  \draw[->, >=stealth, line width=1pt] (2.0, 0.8) -- (3.5, 0.8);
  \node[anchor=south] at (2.75, 1.2) {\large $t$};
  \node[anchor=south] at (-2.75, 1) {\large $K$};

    \begin{scope}[shift={(-1.3,0)}]
        % 下方的弧線 (Smile)
        \draw (-0.6, 0.1) .. controls (-0.3, -0.4) and (0.3, -0.4) .. (0.6, 0.1);
        % 上方的弧線 (Frown)
        \draw (-0.45, -0.05) .. controls (-0.2, 0.25) and (0.2, 0.25) .. (0.45, -0.05);
    \end{scope}

    \draw[red, line width=1.5pt] (4, 0) ellipse (0.25 and 0.5);

\end{tikzpicture} 

\end{center}

The right part extending to infinity as $t \to \infty$. Topologically, this end is a product $\mathbb{R}_+ \times B$, where $B$ is the compact cross-section (in our specific example here, $B = \mathbb{S}^1$). Once we glue those two parts ($K$ and the cylindrical end) together, this \textcolor{cyan}{patching region} will be compact and smooth.

We will focus on the noncompact end. For the actual metric, we just gluing the part for the compact one which comes from the usual compact theory.

 \subsubsection{The Cross-section $B$}

In the context of an asymptotically cylindrical manifold, the cross-section (denoted as $B$) is the compact manifold that specifies the geometry of the infinite end. Geometrically, a manifold $M$ with a cylindrical end is diffeomorphic to a product space $[R, \infty)_t \times B $
where $t$ is the coordinate along the length of the cylinder and $B$ is a compact Riemannian manifold without boundary.

 If we imagine the manifold as a long tube extending to infinity, the cross-section is the slice obtained by taking a constant value of the coordinate $t$. The cross-section $B$ is compact, which allows for a discrete spectral decomposition of operators restricted to it.

The Laplacian on the cylinder $\mathbb{R} \times B$ splits as $\Delta = -\partial_t^2 + \Delta_{h}$ (${h}$ is the metric on $B$), where $\Delta_{h}$ is the Laplacian on the cross-section. The modes of the operator are determined by the eigenvalues $\lambda_j^2$ of $\Delta_{h}$. Each eigenvalue of the cross-section corresponds to a specific decay or growth rate for functions in the cylindrical end.

 In specific example where $B = \mathbb{S}^1$, the cross-section is simply a circle. If the manifold were a 3-D solid tube, the cross-section would be a 2-D surface, such as a sphere or a torus. Regardless of the dimension, the compactness of $B$ is the essential property that enables the Fredholm analysis of $\Delta_g + 1$.
 
 
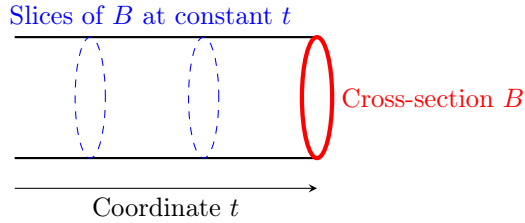
\begin{center}
\begin{tikzpicture}[scale=1.0]
    % Cylindrical part
    \draw[thick] (0,0.8) -- (4,0.8);
    \draw[thick] (0,-0.8) -- (4,-0.8);

    % The cross-section B at different intervals
    \draw[dashed, blue] (1,0) ellipse (0.2 and 0.8);
    \draw[dashed, blue] (2.5,0) ellipse (0.2 and 0.8);
    
    % Highlighted cross-section at the end
    \draw[ultra thick, red] (4,0) ellipse (0.2 and 0.8);
    \node[red, right] at (4.2, 0) {Cross-section $B$};

    % Coordinate t
    \draw[->, >=stealth] (0,-1.2) -- (4,-1.2) node[midway, below] {Coordinate $t$};
    
    % Description
    \node[blue] at (1.8, 1.1) {Slices of $B$ at constant $t$};
\end{tikzpicture}
\end{center}

\subsection{Laplacian in Local Coordinates}
We will show how the Laplacian looks in local coordinates, and we are trying to invert the $\Delta_g+1$. 

\vspace{0.2cm}
\noindent
For our project on an asymptotically cylindrical manifold, the Laplacian can be expressed in terms of $b$-vector fields. Let's see  how our Laplace–Beltrami operator $\Delta_g$ actually looks across different regions of our manifold.

\subsubsection{Laplacian of the Cylindrical End}
On the cylindrical end ($\mathbb{R}_+ \times B$), the manifold exhibits a product space structure. The Laplace-Beltrami operator allows a separation of variables and can be written as:  \[\Delta = -\partial_t^2 + \Delta_{h}\]
and we have the metric
\begin{equation*}
    g_0 = dt^2+dy^2 = (\frac{dx}{x})^2+ dy^2. \tag{product cylindrical metric}
\end{equation*}

\subsubsection{Laplacian of Compact Core}

 In our compact core $K$, it is just standard coordinate definition of the Laplace–Beltrami operator on a Riemannian manifold. Because the interior of our manifold $M$ is completely isolated from the boundary at infinity, the geometry is totally standard there. The operator takes the general coordinate expression for a standard Riemannian metric:
    \[ \Delta_g u = -\frac{1}{\sqrt{|g|}} \partial_i \left( \sqrt{|g|} g^{ij} \partial_j u \right) \]
 We do not need $b$-vector fields in this region. Because this domain is compact, the Rellich-Kondrachov compact embedding theorem can be directly applied to handle remainder terms (after we glue them together) and analyze smoothness. When we glue the front face parametrix (from the cylindrical end) to the interior parametrix from the compact core, it will generate remainder term. Because the compact core is bounded, Rellich–Kondrachov guarantees that Sobolev embeddings are compact. This compact embedding is the exact functional analysis tool used to prove that our remainder terms are compact operators, which is the final step in proving the global operator is Fredholm.

We want to find an inverse. For the compact case, our elliptic estimate will have a compact error. Therefore, we can get the Fredholm property.

\newpage

\section{Review of Compact Case} \label{The Positive Definiteness Method}

\subsection{Proof of Positive Definiteness Implying the Invertibility of $\Delta_g + 1$}
\noindent
Here we want $Q(\Delta_g+1) = I + R$. If $\Delta_g+1$ is not invertible, then -1 is an eigenvalue.

\vspace{0.2cm} \noindent
Let ($M,g$) be a compact Riemannian manifold without boundary. Our goal is to show that the operator \(\Delta_g + 1\) is invertible, which is equivalent to showing that \(-1\) is not an eigenvalue of \(\Delta _{g}\).
\begin{center}
        $\Delta_g  \geq 0, \text{then we have Spec}\Delta_g \subset[0, \infty)\quad$
    since
    \begin{align*}
   \int_{M}\Delta_gf \overline{f}dvol 
    &= \int_{M}|\nabla f|^2dvol \geq 0
    \end{align*}
\end{center}
The invertibility follows directly from the positive semi-definiteness of the Laplacian.

\theoremstyle{definition}
\newtheorem{proposition}{Proposition}[section] % 綁定到章（section）

\begin{proposition}
The Laplacian \(\Delta _{g}\) on a compact Riemannian manifold \(M\) is a positive semi-definite operator, i.e., \(\langle \Delta_g f, f \rangle \geq 0\) for all \(f \in C^\infty(M)\).
\end{proposition}

\begin{proof}
Let \(f \in C^\infty(M)\) be a smooth function on \(M\). We consider the \(L^{2}\) inner product
\[ \langle \Delta _{g}f,f\rangle =\int _{M}(\Delta _{g}f)\overline{f}\,dvol \]
By applying integration by parts, we can shift the derivative. Since \(M\) is a compact manifold without boundary, the boundary term perfectly vanishes (\(\int_{\partial M} = 0\)).

Consequently, we have:\[\int _{M}(\Delta _{g}f)\overline{f}\,dvol=\int _{M}|{}\nabla f|{}^{2}\,dvol\]Since the integrand \(\vert{}\nabla f\vert{}^2\) is nonnegative, the integral over the manifold must also be nonnegative:\[\langle \Delta _{g}f,f\rangle =\int _{M}|{}\nabla f|{}^{2}\,dvol\ge 0 \]This completes the proof.

\end{proof}

Consequently, the spectrum of $\Delta_g$ satisfies
$\operatorname{Spec}(\Delta_g)\subset [0,\infty).
$
Since $M$ is compact, the spectrum of $\Delta_g$ is discrete and consists of eigenvalues
$
0\leq \lambda_0\leq \lambda_1\leq\lambda_2\leq\cdots
$
with $\lambda_j\to\infty$. The eigenvalues of the shifted operator $\Delta_g+1$ are 
$
\lambda_j+1\geq 1.
$
Hence
$
\operatorname{Spec}(\Delta_g+1)
=
\operatorname{Spec}(\Delta_g)+1
\subset [1,\infty).
$
In particular,
\[
0\notin\operatorname{Spec}(\Delta_g+1).
\]
Therefore, $\Delta_g+1$ is invertible.

\newpage

\subsection{Elliptic Estimates for Compact Case}
We first review the compact case, which is basically what we learned in \cite{DyLi24}.

If $M$ is noncompact, we cannot expect the Fredholm property of $\Delta_g$ to hold in general.
This is already evident on the real line : the operator
$\Delta_g = -\partial_t^2$
is not Fredholm as a map $H^2(\mathbb{R}) \to L^2(\mathbb{R})$.
More precisely, the range $-\partial_t^2(H^2(\mathbb{R}))$ is not closed in
$L^2(\mathbb{R})$.

 First we give a quick review of compact case, since the elliptic Fredholm property we are using here is slightly different. It is from the perspective of differential elliptic parametrix.\\ 
\medskip
Let's say we have a compact Riemannian manifold  $(M, g)$ without boundary : 
\[
\Delta_g : H^{s+2}(M) \to H^s(M),    \forall s \in \mathbb{R}  \label{Fredholm Property} \tag{Fredholm Property}
\]
is a Fredholm operator.

\vspace{0.2cm}
\noindent
The Fredholm property means that the kernel is finite dimensional, and the range has finite codimension (cokernel is finite dimensional). The usual way to prove this Fredholm property on Sobolev spaces is to use the elliptic estimate : 
\[
\| u \|_{H^{s+2}}(M) \leq C \| \Delta_g u \|_{H^s} + C_N\| u \|_{H^{-N}}  \label{elliptic estimate1} \tag{elliptic estimate1}
\]  for all $s$ and $N$, with $C_N$ be some constant depends in $N$, and $\| u \|_{H^{-N}}$ means weaker norm in Sobolev space.\\
\medskip  \\ 
What does that elliptic estimate tell us? Notice that if we only have
\[
\| u \|_{H^{s+2}}(M) \leq C \| \Delta_g u \|_{H^s}  
\]
it will be the injectivity of the operator, so \eqref{elliptic estimate1} should be better than injectivity.
Since we have this remainder  $C_N\| u \|_{H^{-N}}$, which is a weaker Soboloev space. Here we can immediately get that the kernel is finite dimensional. We can do this by using the fact that embedding $H^{s+2}(M) \to H^{-N}(M)$ is compact for $s > -N$ in \eqref{Fredholm Property}, we have to use compactness of embeddings in Sobolev spaces. So compactness of embedding plus functional analysis arguments (we should also get an estimate like that on the adjoint, dual of Sobolev spaces, Laplacian and self-adjoint). Once we have this estimate, there's a standard argument to give us the Fredholm property which will not be presented here. 

The key points are the estimate and compactness of embedding. There are many ways to prove \eqref{elliptic estimate1}, one of them is integration by parts, that might be the best way, but here we will prove it via elliptic parametrix.\\
\medskip 
Namely, we will construct a bounded operator : 
\[
Q : H^s({M}) \longrightarrow H^{s+2}({M})
\]
such that

\[
\Delta_g Q = I + R, \qquad
R : H^{-N}_{\text{bdd}}({M}) \longrightarrow  H^{N}({M}).
\]
We don't need such strong property for $R$, but we do get it. So we state it here.

The elliptic parametrix is sort of an inverse. If we didn't have $R$, then $\Delta_g Q = I$, which implies $Q$ is the left inverse of $\Delta_g$, here, we can see that it's an invertible modular smoothing operator, an operator that provides a lot of smoothness.
How to prove the estimate? Just apply it to $u$, we can see that $u$  become 
\[
u = Q\Delta_g u - Ru,
\]
and using the triangle inequality to estimate the norms, we get:
\[
\| u \|_{H^{s+2}} \leq \| Q\Delta_g u \|_{H^{s+2}} + \|Ru\|_{H^{s+2}}.  \label{estimate2} \tag{estimate2}
\]
then we are done here.

\vspace{0.2cm}
\noindent
Now, take a look at \eqref{estimate2}, we see the norm we want ($\| u \|_{H^{s+2}}$ as in \eqref{elliptic estimate1} and \eqref{estimate2}), and here we apply $Q$, which is a bounded operator from $H^s({M})$ to $H^{s+2}({M})$, combine with our smoothing remainder $\|Ru\|_{H^{s+2}}$, we can get \eqref{elliptic estimate1} by using \eqref{estimate2}.

\vspace{0.2cm}
\noindent
We are going to do the parametrix construction in Section 4.

\subsection{Motivation for Parametrix Construction} \label{Difference of two methods}

We should introduce the motivation of doing parametrix construction.  How does it overcome limitations of positive definiteness?
\subsubsection{The Positive Definiteness Method}This method provides an efficient proof of existence and uniqueness. It guarantees invertibility. Despite its simplicity, the method presents significant limitations, that does not give us the explicit inverse. For our specific case, the $\Delta_g +1$ on compact manifold, then we get the invertibilty by integration by parts. However, it has two drawbacks. The first drawbacks is that it does not gives us the explicit inverse, we can't see the mapping properties of inverse. The second drawback is that it does not apply on noncompact manifolds. 
\subsubsection{The Parametrix Construction Method}
The parametrix construction method seeks to define an explicit operator $Q$ such that \[Q(\Delta_g+1) = I + R\] where $R$ acts as a compact term. This is the preferred approach for cylindrical manifolds for several critical reasons. Foremost is the need for asymptotic control, on an asymptotically cylindrical manifold, it is necessary to characterize the precise behavior of solutions as $t \to \infty$. Constructing the parametrix on the exact cylinder provides a concrete model for the exponential decay rate, specifically bounded by $e^{-\sqrt{\lambda_0^2+1}\vert{}t\vert{}}$. Additionally, this construction serves as the standard mechanism within Fredholm theory to prove that an operator is Fredholm when acting on specific function spaces, such as $b$-Sobolev spaces. 

The technique offers robust generalizability. Once the parametrix architecture is established for $\Delta_g+1$, the identical framework can be adapted to analyze more complex geometric operators where basic energy methods are fundamentally inapplicable.

Ultimately, while verifying positive definiteness serves as a useful preliminary diagnostic, fully realizing the objectives of this analysis necessitates the construction of a parametrix. This explicit construction is strictly required to see the mapping properties of $\Delta_g+1$.
\newpage

\section{Separation of Variables} 
As we mentioned in the introduction.
Our manifold is not exactly cylindrical $\mathbb{R} \times B$, it's asymptotically like  $\mathbb{R} \times B$, we have a compact part there. Our metric is not going to be exactly be the product $\Delta_{\mathbb{R} \times B}$, there's also a remainder term $\mathcal{O}$, 
it takes the form
\[
\Delta_g + 1 = \Delta_{B \times \mathbb{R}} + \mathcal{O}
\]
That's why we can only do separation of variable locally. For the actual construction, we need to deal with that remainder part, which motivates us to do the parametrix construction in Section 4.

In \cite{DyLi24}, they work on the example $\Delta_g$ on exact cylinder $M = \mathbb{R}_t \times \mathbb{S}_y^1$ which is not invertible. For convenience, they ignore the zero mode. Here we use $\Delta_g+1$ as an example and kind of redo the separation of variables calculation they show in \cite{DyLi24} but with different parameter.

Now let's see the part where we do the separation of variables.
Given an asymptotical cylindrical manifold where the metric near the end is exponentially 
close to a product $\mathbb{R} \times \mathbb{S}^1$, prove the invertibility of 
$\Delta_g + 1$ by constructing the parametrix, which is an approximate inverse. 
Later one can generalize the metric to exponentially close to $\mathbb{R} \times B$ 
where B is any compact manifold.

\subsection{Fourier Decomposition on the Exact Cylinder $\mathbb{R} \times \mathbb{S}^1$} \label{Fourier Decomposition}
\noindent
The operator $\Delta_g + 1$ on the end still looks like:$$ \Delta_g + 1= -\partial_t^2 + \Delta_y + 1$$
We can still ``decompose'' any function $u$ on the end into $f(t, y) = \sum f_k(r)\phi_k(y)$. The problem of inverting the global operator reduces to solving a family of 1-D ordinary differential equations for each $f_k(t)$.\\
For the exact cylinder $M = \mathbb{R}_t \times \mathbb{S}_y^1$, we fix the product metric first, $g_0$ is just the product metric, and in the compactified coordinates, it has the form $ \frac{dx^2}{x^2} + dy^2$.
\begin{equation*}
    g_0 = dt^2+dy^2 = (\frac{dx}{x})^2+ dy^2. \tag{product cylindrical metric} \label{product cylindrical metric}
\end{equation*}
(In general, we should allow $g$ is of the form $g$ = $g(x,y,\frac{dx}{x}, dy)$, any metric with $b$-Structure is fine.)
On this manifold, we also have the function space : the $b$-Sobolev space.\\
So our metric is smooth up to the boundary, now I have a metric on the manifold with boundary which is perfectly smooth except this $x$ in $\frac{dx}{x}$. For example, we have the product metric :
\[
 \frac{dx^2}{x^2}+h(y,dy)
 \tag{product metric}
\]
this product metric is a basic case. Compactifing it at the end.

Notice that $\Delta_y$ is $\Delta_g$ in $y$ variable.  The Laplacian is given by $\Delta_{g_0} = -\partial_t^2 + \Delta_y+1$. We want to invert the operator, so we need to show $\Delta_y+1$ is invertible. As for $-\partial_t^2$, we will do Fourier transform in $t$.

\noindent
We can decompose any function $f(t,y)$ into its Fourier modes:
    $$f(t,y) = \sum_{k \in \mathbb{Z}} f_k(t) e^{iky}$$

\noindent
Applying $\Delta_{g_0} + 1$ to the $k$-th mode make it to be a 1-D operator:
    $$P_k = -\partial_t^2 + k^2 + 1$$
    
\noindent The integral kernel of the inverse for $P_k$ on $\mathbb{R}$ is exactly:
    $$R_k(t,t') = \frac{e^{-\sqrt{k^2+1}|t-t'|}}{2\sqrt{k^2+1}}$$
    
\noindent Summing over all Fourier modes $k \in \mathbb{Z}$, the Schwartz kernel of the exact inverse $Q$ for $\Delta_{g_0} + 1$ becomes:
    $$Q(t,t',y,y') = \sum_{k \in \mathbb{Z}} \frac{e^{-\sqrt{k^2+1}|t-t'|}}{2\sqrt{k^2+1}} e^{ik(y-y')}$$
    This gives us a bounded operator that satisfies $(\Delta_{g_0} + 1) Q = I$.

The compact core $K$ does not possess a product structure, preventing a global separation of variables.

\begin{proposition}
Let $M$ be a manifold with asymptotically cylindrical end, where the metric $g$ approaches the exact product metric $dt^2 + g_B$ up to an exponentially decaying error. Let $P = \Delta_g + 1$ be the shifted Laplace operator, and let $Q$ be a global parametrix constructed from the exact translation-invariant model $-\partial_t^2 + \Delta_B + 1$. Applying $P$ to the parametrix yields
\[
    P Q = \mathrm{Id} - \mathbf{R}.
\]
Then the following structural properties hold for the inversion:

\label{prop1}
\end{proposition}
 The remainder $\mathbf{R}$ is a compact, smoothing operator that captures the geometric deviation from the exact cylinder, decaying exponentially at the cylindrical end.

Because $\mathbf{R}$ is compact and exponentially decaying, the error term can be directly inverted, making the operator $\mathrm{Id} - \mathbf{R}$ invertible on the associated function spaces.
    
The exact inverse of the shifted operator exists and is given explicitly by the composition
    \[
        (\Delta_g + 1)^{-1} = Q (\mathrm{Id} - \mathbf{R})^{-1}.
    \]
 The true inverse and the model parametrix $Q$ are exponentially close to each other. Their difference, given by
    \[
        (\Delta_g + 1)^{-1} - Q = Q \mathbf{R} (\mathrm{Id} - \mathbf{R})^{-1},
    \]
    is strictly an exponentially decaying smoothing operator.

\subsection{Generalizing to the $\mathbb{R} \times B$}

 \subsubsection{Basic Settings of Generalization to $\mathbb{R} \times B$}

Here is how the information from that cylindrical end diagram generalizes to the case of $B$ is any compact manifold:
The separation of variables remains universal in $\mathbb{S}^1$ case, using a standard Fourier series. When $B$ is a general compact manifold, the Fourier decomposition generalizes to the spectral decomposition of the Laplacian on $B$ (Because $B$ is compact, $\Delta_B$ is a well-behaved elliptic operator).\\
On $\mathbb{S}^1$, we have eigenvalues $n^2$ and eigenfunctions $e^{in\theta}$. On a compact manifold $B$, the Laplacian $\Delta_B$ has a discrete spectrum $0 = \lambda_0 < \lambda_1 \le \lambda_2 \dots \to \infty$ with corresponding smooth eigenfunctions $\phi_j$.

Now, let the cross-section be an arbitrary compact Riemannian manifold $(B, h)$, making the cylinder $M = \mathbb{R}_t \times B_y$ with metric $g = dt^2 + h(y, dy)$. The operator is $\Delta_{g} + 1 = -\partial_t^2 + \Delta_y + 1$.\\
    Instead of standard Fourier modes, we use the spectral decomposition of the Laplacian $\Delta_h$ on the compact manifold $B$. Let $\{\phi_j(y)\}$ be an orthonormal basis of $L^2(B)$ consisting of eigenfunctions such that $\Delta_h \phi_j = \lambda_j^2 \phi_j$, with eigenvalues $\lambda_j^2 \ge 0$.\\
    The operator acts on the $j$-th eigenspace as $-\partial_t^2 + \lambda_j^2 + 1$. Using the same 1-D ODE logic, the exact inverse kernel on $\mathbb{R} \times B$ becomes:
    $$Q(t,t',y,y') = \sum_{j} \frac{e^{-\sqrt{\lambda_j^2+1}|t-t'|}}{2\sqrt{\lambda_j^2+1}} \phi_j(y) \overline{\phi_j(y')}$$
    The strict positivity of $\lambda_j^2 + 1 \ge 1$ ensures that exponential decay is maintained for all modes, meaning the operator is invertible and Fredholm natively on $L^2$.

\subsubsection{Details of the Generalization to $\mathbb{R} \times B$}

\noindent
To construct the Schwartz kernel $Q(t, t', y, y')$ for the exact inverse of $\Delta_{g_0} + 1$ on the generalized cylinder $M = \mathbb{R}_t \times B_y$, we solve the distributional equation:
\begin{equation*}
    (-\partial_t^2 + \Delta_h + 1) Q(t, t', y, y') = \delta(t - t') \delta_B(y, y')
\end{equation*}
where $\Delta_h$ is the Laplace-Beltrami operator on the compact Riemannian manifold $(B, h)$, and $\delta_B$ is the Dirac distribution on $B$.

\noindent
Here's the spectral decomposition on $B$.
Since $B$ is a compact Riemannian manifold, $\Delta_h$ has a discrete spectrum. Let $\{\phi_j(y)\}_{j=0}^\infty$ be an orthonormal basis for $L^2(B)$ consisting of eigenfunctions:
\begin{equation*}
    \Delta_h \phi_j = \lambda_j^2 \phi_j, \quad 0 \leq \lambda_0^2 \leq \lambda_1^2 \leq \dots \to \infty
\end{equation*}
The delta distribution on $B$ can be represented as:
\begin{equation*}
    \delta_B(y, y') = \sum_{j=0}^\infty \phi_j(y) \overline{\phi_j(y')}
\end{equation*}
We expand the full kernel $Q$
\begin{equation*}
    Q(t, t', y, y') = \sum_{j=0}^\infty R_j(t, t') \phi_j(y) \overline{\phi_j(y')}
\end{equation*}

\noindent
Let's solve the 1-D ODE.
Applying $-\partial_t^2 + \Delta_h + 1$ to the series and matching coefficients for each mode $j$ yields a family of decoupled ODEs for the 1-D factor $R_j$:
\begin{equation*}
    (-\partial_t^2 + \lambda_j^2 + 1) R_j(t, t') = \delta(t - t')
\end{equation*}
Defining $\omega_j = \sqrt{\lambda_j^2 + 1}$, the equation becomes $(-\partial_t^2 + \omega_j^2) R_j = \delta(t - t')$.\\

\vspace{0.2cm}
\noindent
Let's solve $R_j(t, t')$.
For decay and continuity, we want to ensure $Q$ is a bounded operator on $L^2(M)$, we require $R_j \to 0$ as $|t - t'| \to \infty$. Combined with the continuity requirement at $t = t'$, we have:
\begin{equation*}
    R_j(t, t') = A e^{-\omega_j |t - t'|}
\end{equation*}

\noindent
Integrating the ODE across the singularity $t = t'$ gives the first derivative:
\begin{equation*}
    -\left[ \lim_{t \searrow t'} \partial_t R_j - \lim_{t \nearrow t'} \partial_t R_j \right] = 1 
    \implies -(-A\omega_j - A\omega_j) = 1 \implies A = \frac{1}{2\omega_j}
\end{equation*}
Substituting $\omega_j$ back in:
\begin{equation*}
    R_j(t, t') = \frac{e^{-\sqrt{\lambda_j^2+1} |t - t'|}}{2\sqrt{\lambda_j^2+1}}
\end{equation*}
As for the final inverse kernel,
summing over all $j$, the Schwartz kernel of the exact inverse $Q = (\Delta_{g_0} + 1)^{-1}$ on $\mathbb{R} \times B$ is:
\begin{equation*}
    Q(t, t', y, y') = \sum_{j=0}^\infty \frac{e^{-\sqrt{\lambda_j^2+1}|t - t'|}}{2\sqrt{\lambda_j^2+1}} \phi_j(y) \overline{\phi_j(y')}
\end{equation*}
The strict positivity $\lambda_j^2 + 1 \geq 1$ ensures uniform exponential decay for all modes, making the operator natively Fredholm on $L^2$.

The powerful part of microlocal analysis is that the metric does not have to be exactly the product structure. If it is in a product form, the separation of variables done very well. Our current work is just replacing $\mathbb{S}^1$ to  $B$. The advantage of using parametrix construction is that we don't care if it was exact product, it can have lower order terms, it still works. The fact that $B$ is any compact manifold (rather than just $\mathbb{S}^1$) is exactly why the $b$-calculus framework is powerful. It doesn't care about the specific shape of $B$, as long as $B$ is compact.

\newpage

\section{Parametrix Construction}
\subsection{The Blown-up Double Space}

We will follow the structure of Section 2.5 of \cite{DyLi24} (where we see the blown-up $b$-space argument), and the Schwartz kernel construction of the $b$-space. Our goal is to write down the parametrix construction on the $b$-space, the argument will be quite similar. It's more similar compare to the separation of variables in some sense. Just with a different operator, now it's $\Delta+1$ instead of $\Delta$. 

Same for the mapping properties. Once we have the parametrix construction of the mapping on the double space (the kernel on $[\overline{M} \times \overline{M}; Y \times Y]$), the mapping properties actually follow from the general $b$-operator mapping theory, Melrose’s $b$-calculus notes will give us the mapping properties on weighted $b$-Sobolev spaces. There is no need to reinvent the wheel there. The reader can also see the book \cite{Mel96} for more details in blow-ups. \cite{Gri01} is also a nice reference in blow-up geometry.

The point is that we have to construct the $b$-parametrix since it is not apparent from the previous separation of variables. Again, the actual construction is similar to the \cite{DyLi24}. We already did the part that is different, we are kind of redoing some of the work that is less cheating since the operator in \cite{DyLi24} is $\Delta$, that will be easier to do the summation of the expansion/series, so we can see that more clearly, but we should worry about the problem of invertibility. As for the case of $\Delta+1$, we have an actual inverse instead of inverse modulo zero. 

\noindent
Transitioning from the exact separation of variables on the cylinder to the general parametrix on the $b$-double space is exactly what we are trying to do. We are trying to capture the asymptotic behaviors.

\vspace{0.2cm}
\noindent
By the Schwartz kernel theorem, a linear operator $Q$ acting on functions on $M$ can be represented by an integral kernel $\mathbf{Q}(y, y')$ living on the product space $M \times M$. When we compactify the ends of our cylinder, this product space becomes $\overline{M} \times \overline{M}$. However, as we saw with the $\frac{x'}{x}$, standard calculus breaks down at the corner $Y \times Y$. To fix this, we perform the geometric blow-up to create the $b$-double space, denoted $[\overline{M} \times \overline{M}; Y \times Y]$ or $M^2_b$.
Now we see more coordinates. Discussing Region I, Region II, and Region III.
See the following figure (or Figure B) for an illustration of these coordinates.

\vspace{0.2cm} \noindent
We want to discuss polyhomogeneous expansions on boundary faces. We can see the details of Mellin transform in Section 4.7 of \cite{Mel96} and the polyhomogeneous conormal distributions in Section 4.13 of \cite{Mel96}.
When we talk about a polyhomogeneous kernel on the blown-up space, we take the Schwartz kernel $\mathbf{Q}$, lifting it to this new geometry $[\overline{M} \times \overline{M}; Y \times Y]$, analyzing its behavior there, where the boundary coordinates are actually smooth.

 \vspace{0.2cm} \noindent
 The great news is that our operator, $\Delta_g + 1$, makes the $b$-space construction easier to see than the Laplacian $\Delta_g$ case in \cite{DyLi24}. Here, we don't have to ``illegally ignore'' any zero mode because the $+1$ shift mathematically eliminates the obstruction.

 Here, when we blow up the double space to construct $[\overline{M} \times \overline{M}; Y \times Y]$, we will first discuss the front face. Geometrically, restricting to front face means taking $r = e^{-t} \to 0$ while keeping the $s =\frac{ x'}{x} = e^{t-t'}$ finite. At front face, all the perturbation terms $O(e^{-t})$ drop off as it shown in \cite{DyLi24} equation (3.4). The operator freezes into the exact, translation invariant model operator $-\partial_t^2 + \Delta_h + 1$.

The philosophy of the blow-up is to trade geometric complexity for analytic simplicity. So we can see that in the blowup picture, the analysis looks a lot more simpler, and the geometry becomes more complicated. The family of smooth functions up to the boundary in the blown-up space is much larger than the family of smooth functions up to the boundary. So the function space has changed.

\vspace{0.2cm} \noindent
Blowing up the diagonal in Figure A, we obtain the following three regions in Figure B.

\begin{figure}[htbp]
\makebox[\textwidth][c]{
    % =========================================================
    % Diagram A
    % =========================================================
    \begin{minipage}[c]{0.48\textwidth}
        \centering
        \begin{tikzpicture}[>=stealth, scale=1.1] 
       
            \draw[thick] (0,4) -- (0,0);
            \draw[thick] (0,0) -- (4,0);
            \node at (3,2.1) {diagonal};
            % --- Dashed Region Divider ---
            \draw[thick] (0,0) -- (3.5,3.5);
            % --- Coordinate Patches ---
            \begin{scope}[shift={(0.2,2.5)}]
            \draw[->, thick] (0,0.1) -- (0,1) node[right] {$x'$};
            \end{scope}
            \begin{scope}[shift={(2.5,0.2)}]
            \draw[->, thick] (0.1,0) -- (1,0) node[above] {$x$};
            \end{scope}
        \end{tikzpicture}
        \par\vspace{5pt}
        \small\textbf{Figure A}
    \end{minipage}%
    \hfill
    % =========================================================
    % Diagram B
    % =========================================================
    \begin{minipage}[c]{0.6\textwidth}
        \centering
        \begin{tikzpicture}[line width=0.8pt, scale=0.9] % 稍微縮小比例避免溢出
            % --- Main Boundaries ---
            \draw (0,5) -- (0,2) node[midway, left=2pt] {left boundary};
            \draw (0,2) arc (90:0:2) node[pos=0.5, below left=2pt] {front face ($r=0$)};
      
            \draw (2,0) -- (6,0) node[midway, below=2pt] {right boundary};
            % --- Region Dividers ---
            \draw[dashed] ({2*cos(70)}, {2*sin(70)}) -- ++(75:3.5);
            \draw[dashed] ({2*cos(25)}, {2*sin(25)}) -- ++(15:4);
            % --- Region Labels ---
            \node at (0.75,5.2) {Region III};
            \node at (2.75,4) {Region I};
            \node at (4,3) {Region I};
            
    \draw (1.425, 1.425) -- (5,5); 
            \node at (4.8,1) {Region II};
            \node at (5,4) {diagonal};
            % --- Local Coordinate Systems ---
            \begin{scope}[shift={(0.1,3)}]
            \draw[->] (0,0) -- (0,1) node[right] {$x'$};
            \draw[->] (0,0) arc (90:55:1) node[above] {$s'$};
            \end{scope}
            \begin{scope}[shift={(2.2,0.2)}]
            \draw[->] (0,0) -- (1,0) node[above] {$x$};
            \draw[->] (0,0) arc (0:35:1) node[right] {$s$};
            \end{scope}
                \begin{scope}[shift={({2.35*cos(47.5)}, {2.35*sin(47.5)})}, rotate=47.5]
        \draw[->] (0,0) -- (1, 0) node[above] {$r$};
        \draw[->] (0,0) arc (0:35:1) node[above] {$\tau$};
    \end{scope}
        \end{tikzpicture}
        \par\vspace{5pt}
        \small\textbf{Figure B}
    \end{minipage}%
} % 結束 makebox
\end{figure}

\vspace{0.2cm} \noindent
\textbf{Region I} covers the interior of the front face where $r = e^{-t}$ approaches 0. In this region, the ratio $s = \frac{x'}{x} = e^{t-t'}$ remains finite, which correctly captures the space containing the diagonal.

\vspace{0.2cm} \noindent
\textbf{Region II} isolates the behavior near the right boundary. The coordinate $s$ evaluates the regime where $s$ approaches 0.

\vspace{0.2cm} \noindent
\textbf{Region III} captures the behavior near the left boundary utilizing the coordinate $s'$. This models the opposite regime where $s'$ approaches 0.

\vspace{0.2cm} \noindent
More generally, it is crucial that $\mathbf{Q}(\tau, r, y, y')$ is smooth in $r$ up to $r = 0$ as a family of conormal distributions in $(\tau, y, y')$ with respect to $\{y = y'\}$.

 \subsection{The Indicial Roots}

 The indicial roots of the $b$-differential operator will determine the mapping properties of the parametrix. For the operator $\Delta + 1$, the roots are $\pm \sqrt{\lambda_j^2 + 1}$. To say a metric is exponentially close to $\mathbb{R} \times B$ means that as the coordinate $t$ increases, the metric $g$ on $M$ approaches the product metric $dt^2 + h$ . The rate of this convergence is what allows the global parametrix to have a compact remainder term.
 
This part explains why the structure of $\Delta_g + 1$ is easier to see.
 In the $b$-calculus, the asymptotic behavior of the kernel is dictated by the indicial operator at the boundary. For $\Delta_g$, \cite{DyLi24} showed the indicial operator is $I(\lambda) = \Delta_h - \lambda^2$. The roots are $\lambda = \pm \mu_j$ (where $\mu_j^2$ are the eigenvalues of the cross-section $Y$). The zero eigenvalue ($\mu_0 = 0$) creates a root at $\lambda = 0$, which ruins the Fredholm property on standard $L^2$. For our operator, $\Delta_g + 1$, the indicial operator becomes:\[I_{+1}(\lambda) = \Delta_h - \lambda^2 + 1 \tag{indicial operator} \label{indicial operator}
 \]
 The indicial roots are the solutions to $I_{+1}(\lambda) = 0$, which gives:\[
 \lambda_{j, \pm} = \pm \sqrt{\mu_j^2 + 1} \tag{indicial roots} \label{indicial roots}\]
 Because the spectrum of $\Delta_h$ is nonnegative ($\mu_j^2 \ge 0$), all indicial roots are strictly bounded away from zero ($\vert{}\lambda\vert{} \ge 1$). The entire spectral strip $-1 < \text{Re}(\lambda) < 1$ is completely free of roots. This means our operator is naturally Fredholm (and invertible) on standard unweighted spaces ($\alpha = 0$).
 
 \subsection{Kernel on the Front Face}
 In the blown-up double space, the front face geometrically represents the exact cylindrical limit at infinity. The kernel restricted to the front face, $\mathbf{Q}_{\text{ff}}$, is where our separation of variables calculation lives.
 
 To construct $\mathbf{Q}$ near ff (Region II and III), we use the inverse Mellin transform of the indicial operator:$$\mathbf{Q}_{\text{ff}} = \frac{1}{2\pi i} \int_{\text{Re}(\lambda) = \alpha} s^\lambda I_{+1}(\lambda)^{-1} \, d\lambda$$Because we have no roots near $0$, we can integrate straight down the imaginary axis ($\alpha = 0$). When we evaluate this integral using the residue theorem, we pick up the poles at the indicial roots $\lambda = \sqrt{\mu_j^2 + 1}$. This residue evaluation reconstructs our 1-D Green's function from the separation of variables: the $s^{\sqrt{\mu_j^2 + 1}}$ term in the $b$-coordinates is exactly the $e^{-\sqrt{\mu_j^2+1}\vert{}t-t'\vert{}}$ decay we found.
 
 \subsection{Asymptotics : Polyhomogeneity at the Side Boundaries}
 
 To say ``asymptotically, it has this behavior'' we need to describe the kernel $\mathbf{Q}$ as a polyhomogeneous conormal distribution on the blown-up space.

 \begin{itemize}
     \item [] Near the right boundary (rb $\cap$ ff)
      The coordinate is $s = \frac{x'}{x} \to 0$. The kernel will have an asymptotic expansion in powers of $s$. Because our indicial roots are $\sqrt{\mu_j^2 + 1}$, the expansion starts at $s^1$ (or $s^{\sqrt{\mu_0^2+1}}$).

    \item [] Near the left boundary (lb $\cap$ ff)
    The coordinate is $s' = \frac{x}{x'} \to 0$. The expansion starts at $(s')^{\sqrt{\mu_0^2+1}}$.

 \end{itemize}
 We can state that the parametrix $\mathbf{Q}$ is smooth in the interior, it has a standard pseudodifferential singularity along the lifted diagonal, and is polyhomogeneous at the side faces (lb and rb) with index sets strictly bounded.

Once we write down that the kernel $\mathbf{Q}$ has these polyhomogeneous bounds on $[\overline{M} \times \overline{M}; Y \times Y]$, we are done with the hard part. We can simply state that by the calculus of $b$-operators (see Melrose's note), the fact that $\mathbf{Q}$ is polyhomogeneous with index sets strictly greater than $0$ implies it acts as a bounded operator $L^2(M) \to H^2_b(M)$.

\subsection{Residue Theorem Calculation}

In this section, we present the residue theorem calculation that connects the Mellin inverse contour integral directly to our $e^{-\sqrt{\mu_j^2+1}\vert{}t-t'\vert{}}$ separation of variables sum, so that we can have some concrete algebraic link for the paper. We will see the front face Mellin inversion and residue calculation.

\subsubsection{The Mellin Integral Representation}

On the front face of the blown-up double space $[\overline{M} \times \overline{M}; Y \times Y]$, the normal operator of $\Delta_g + 1$ is the model cylinder operator $-\partial_t^2 + \Delta_h + 1$.

\vspace{0.2cm} \noindent
In the $b$-coordinates, the boundary defining functions are related to the cylindrical coordinate via $x = e^{-t}$ and $x' = e^{-t'}$. The relative coordinate $s = \frac{x'}{x} \in (0, \infty)$ satisfies:
\[s = e^{t - t'}, \quad \text{or equivalently} \quad \log s = t - t' \]
The front face kernel $\mathbf{Q}_{\text{ff}}(s, y, y')$ is obtained by taking the Mellin inverse transform of the inverse of the indicial operator $I_{+1}(\lambda)^{-1}$:
\[ \mathbf{Q}_{\text{ff}}(s, y, y') = \frac{1}{2\pi i} \int_{\text{Re}(\lambda) = 0} s^\lambda I_{+1}(\lambda)^{-1}(y, y') \, d\lambda 
\]
(By integrating the inverse indicial operator along the imaginary axis, we successfully generated the local representation of the conormal distribution at the front face.)

Recall that the indicial operator acts on the orthonormal eigenbasis $\{\phi_j(y)\}_{j=0}^\infty$ of $\Delta_h$ with eigenvalues $\mu_j^2 \geq 0$ as:
\[ I_{+1}(\lambda) \phi_j = (\Delta_h - \lambda^2 + 1)\phi_j = (\mu_j^2 + 1 - \lambda^2)\phi_j \]
Expanding $I_{+1}(\lambda)^{-1}$ in the spectral basis yields:
\[ I_{+1}(\lambda)^{-1}(y, y') = \sum_{j=0}^\infty \frac{\phi_j(y) \overline{\phi_j(y')}}{\mu_j^2 + 1 - \lambda^2} \]
Substituting this spectral expansion into the Mellin inverse integral and setting $\omega_j = \sqrt{\mu_j^2 + 1} \geq 1$, we interchange the sum and integration:
 
\[ \mathbf{Q}_{\text{ff}}(s, y, y') = \sum_{j=0}^\infty \phi_j(y) \overline{\phi_j(y')} \left( \frac{1}{2\pi i} \int_{\text{Re}(\lambda) = 0} \frac{s^\lambda}{\omega_j^2 - \lambda^2} \, d\lambda \right) \]

\subsubsection{Evaluating the Contour Integral via Residues}

Let \[J_j(s) = \frac{1}{2\pi i} \int \frac{s^\lambda}{\omega_j^2 - \lambda^2} \, d\lambda\]
We rewrite the integrand as:
 
\[ f(\lambda) = \frac{s^\lambda}{\omega_j^2 - \lambda^2} = -\frac{s^\lambda}{(\lambda - \omega_j)(\lambda + \omega_j)} \]
The integrand has two simple poles at $\lambda = \pm \omega_j = \pm \sqrt{\mu_j^2 + 1}$. Since $\omega_j \geq 1$, both poles lie off the imaginary axis $\text{Re}(\lambda) = 0$.

\vspace{0.2cm} \noindent
Note that $s^\lambda = e^{\lambda \log s} = e^{\lambda(t - t')}$. The choice of closing the integration contour in the complex plane depends on the sign of $\log s = t - t'$:

\subsubsection*{Case 1: $t > t'$ (i.e., $s > 1 \implies \log s > 0$)}
To ensure that the exponential term $e^{\lambda \log s}$ decays along the large arc $C_R$ as $R \to \infty$, we require $\text{Re}(\lambda) < 0$. Therefore, we close the contour in the left half-plane ($\text{Re}(\lambda) \leq 0$) using a counterclockwise semicircle $\gamma_L$:
\begin{itemize}
    \item [-] Enclosed pole: $\lambda = -\omega_j$.
    \item [-] By Jordan's Lemma, the integral over the arc $\gamma_L$ vanishes as $R \to \infty$.
\end{itemize}
By Cauchy's residue theorem:
\[ J_j(s) = \text{Res}_{\lambda = -\omega_j} f(\lambda) = \lim_{\lambda \to -\omega_j} (\lambda + \omega_j) \left( -\frac{s^\lambda}{(\lambda - \omega_j)(\lambda + \omega_j)} \right) = \frac{s^{-\omega_j}}{2\omega_j} \]
Substituting $s = e^{t - t'}$ into $s^{-\omega_j}$:
 
\[ J_j(s) = \frac{e^{-\omega_j(t - t')}}{2\omega_j} = \frac{e^{-\sqrt{\mu_j^2 + 1}|t - t'|}}{2\sqrt{\mu_j^2 + 1}} \quad (\text{since } t - t' > 0) \]

\subsubsection*{Case 2: $t < t'$ (i.e., $s < 1 \implies \log s < 0$)}
To ensure decay as $R \to \infty$, we require $\text{Re}(\lambda) > 0$. Therefore, we close the contour in the right half-plane ($\text{Re}(\lambda) \geq 0$) using a clockwise semicircle $\gamma_R$:
\begin{itemize}
    \item [-] Enclosed pole: $\lambda = +\omega_j$.
    \item [-] The clockwise orientation introduces a negative sign: $J_j(s) = -\text{Res}_{\lambda = +\omega_j} f(\lambda)$.
\end{itemize}
Computing the residue:
\[ \text{Res}_{\lambda = +\omega_j} f(\lambda) = \lim_{\lambda \to +\omega_j} (\lambda - \omega_j) \left( -\frac{s^\lambda}{(\lambda - \omega_j)(\lambda + \omega_j)} \right) = -\frac{s^{\omega_j}}{2\omega_j} \]
Thus:
\[ J_j(s) = -\left( -\frac{s^{\omega_j}}{2\omega_j} \right) = \frac{s^{\omega_j}}{2\omega_j} \]
Substituting $s = e^{t - t'}$ into $s^{\omega_j}$:
 
\[ J_j(s) = \frac{e^{\omega_j(t - t')}}{2\omega_j} = \frac{e^{-\sqrt{\mu_j^2 + 1}|t - t'|}}{2\sqrt{\mu_j^2 + 1}} \quad (\text{since } t - t' < 0 \implies t - t' = -|t - t'|) \]

\subsubsection{Reassembling the Front Face Parametrix}

Combining both cases, the 1-D Mellin inverse integral yields :
 
\[ J_j(s) = \frac{e^{-\sqrt{\mu_j^2 + 1}|t - t'|}}{2\sqrt{\mu_j^2 + 1}} \quad \forall t, t' \in \mathbb{R} \]
 Substituting $J_j(s)$ back into the expression for $\mathbf{Q}_{\text{ff}}$ completes the identity:
\[\mathbf{Q}_{\text{ff}}(t, t', y, y') = \sum_{j=0}^\infty \frac{e^{-\sqrt{\mu_j^2 + 1}|t - t'|}}{2\sqrt{\mu_j^2 + 1}} \phi_j(y) \overline{\phi_j(y')} \]
 This proves that the inverse of the indicial operator integrated along the Mellin contour on the front face matches the exact 1-D Green's function derived via separation of variables.

 Separation of variables is exact on the front face, but the front face is only a single ``slice'' of the geometric and analytical picture.
 When we restrict the problem to the front face, the geometry is the exact model cylinder $\mathbb{R} \times Y$. Because we are generalizing to a compact cross-section $Y$, the cross-sectional Laplacian $\Delta_h$ has a discrete spectrum of eigenfunctions $\phi_j(y)$ with eigenvalues $\mu_j^2$.

\subsection{Structural Analogy Between $\Delta_g$ and $\Delta_g + 1$}
The approach here is analogous to that in \cite{DyLi24}.

\renewcommand{\arraystretch}{1.5}
\begin{table}[h]
\centering
\begin{tabularx}{\textwidth}{>{\bfseries}p{0.18\textwidth}X X}
\hline
Step &
Dyatlov \& Li ($\Delta_g$) &
Our Project ($\Delta_g + 1$) \\
\hline

Reduction to ff &
Restrict to $r=0$ because perturbation terms containing $r\partial_r$ vanish at the front face (Eq.\ 3.4  in \cite{DyLi24}). &
Identical: $r\partial_r$ terms vanish at $r=0$, leaving the translation-invariant model operator. \\

\hline
Equation of ff &
\[
(-\partial_\tau^{2}+\Delta_y)\mathbf{Q}(\tau,y,y')
=\delta(\tau)\delta(y-y')
\]
(Eq.\ 3.3  in \cite{DyLi24}). &
\[
(-\partial_\tau^{2}+\Delta_y+1)\mathbf{Q}(\tau,y,y')
=\delta(\tau)\delta(y-y').
\]
\\

\hline
Dual Inversion &
Takes Fourier transform in $\tau=t-t'$ with dual variable $\sigma$ (Eq.\ 3.5  in \cite{DyLi24}). &
Takes Mellin transform in $\tau$ with dual variable $\lambda=i\sigma$. \\

\hline
Pole Locations &
Poles at $\sigma=\pm i\mu_j$. &
Poles shifted to $\sigma=\pm i\sqrt{\mu_j^2+1}$. \\

\hline
Integration Contour &
We must shift contour to $\mathbb{R}-ia$ because $\sigma=0$ is a pole (zero mode). &
No shift needed. Integration line is the real axis ($\operatorname{Im}(\sigma)=0$), since poles are bounded away from zero. \\

\hline
                                                    \end{tabularx}
\end{table}

\vspace{0.2cm}
\noindent
In Figure 4 of \cite{DyLi24}, we can see red dots lying directly on the real line $\text{Im}(\sigma) = 0$. This corresponds to the zero eigenvalue $\mu_0 = 0$ of the cross-sectional Laplacian $\Delta_y$. Because of this pole at $\sigma = 0$, they cannot integrate along the real line $\mathbb{R}$, they are forced to shift the contour to $\mathbb{R} - ia$ (the blue solid line). This shift is the reason why $\Delta_g$ is not Fredholm on standard unweighted $L^2$.

Unlike their setup in equation (3.5) of \cite{DyLi24}, our spectral shift $+1$ places the contour directly on the real axis without needing to avoid a zero mode pole at $\sigma = 0$.

Separation of variables gives the explicit formula for the model problem on an infinite, unperturbed cylinder. Mellin Inversion on $\text{ff}$ proves that the $b$-calculus machinery correctly reproduces this model solution on the front face of the blown-up space.The $b$-calculus takes that front face model $\mathbf{Q}_{\text{ff}}$, glues it to an interior elliptic parametrix $\mathbf{Q}_{\text{int}}$, handles the nonexact metric perturbations, and uses the residual decay of $\mathbf{R}$ to prove global $L^2(M) \to H^2_b(M)$ invertibility.

\vspace{1em} \noindent
In Melrose's book, everything is done in the $s$ variable (we can work on either $\tau$ or $s$ variable). So instead of Fourier transform, we take the Mellin transform, exponentiated version.

\newpage

 We have the $b$-Sobolev spaces$$H^s_b(M), \quad s \in \mathbb{R}$$ with metric $g$ (or more precisely, to the class of $b$-metrics). If $s \ge 0$ is an integer, we can define $b$-vector fields:$$H^s_b(M) = \{ f : V_1 \dots V_m f \in L^2(M; d\text{vol}_g) \text{ for all } V_1, \dots, V_m \in \mathcal{V}_b(M), m \le s \}.$$For example, on the exact cylinder $\mathbb{R}_t \times Y$ with the metric $dt^2 + h_y$, this just corresponds to applying the fields $\partial_t, \partial_y$ no more than $s$ times and requiring the result to remain in $L^2(\mathbb{R} \times Y)$. The Fredholm property of the Laplacian $\Delta_g$ does not hold in general on these unweighted spaces. The case of the real line is an example: the operator $-\partial_t^2$ is not Fredholm $H^2(\mathbb{R}) \to L^2(\mathbb{R})$ because its range $-\partial^2_t(H^1(\mathbb{R}))$ is not closed in $L^2(\mathbb{R})$. This failure stems directly from the existence of a double zero root in the indicial spectrum.
 
 To establish the Fredholm property, one typically (like what we see in the \cite{DyLi24}) introduces weighted $b$-Sobolev spaces$$x^\alpha H^s_b(M),$$where the order $\alpha \in \mathbb{R}$ of the space dictates the allowable growth or decay at infinity. For example, on the exact cylindrical end $x = e^{-t}$, this space corresponds to $e^{\alpha t}H^s(\mathbb{R} \times Y)$. To determine the exact conditions on $\alpha$ required for invertibility, we look for solutions to the asymptotic equation $(\Delta_g + 1)u = 0$ in the infinite end which have the form$$u(x, y) = x^\lambda v(y) = e^{-\lambda t}v(y).$$Here $\lambda \in \mathbb{C}$ is a constant. The action of our shifted operator $\Delta_g + 1$ on such functions is described by its indicial operator, which is a $\lambda$-dependent family of second-order elliptic partial differential operators $I_{+1}(\lambda)$ on the boundary $Y = \partial \overline{M}$ defined as follows:$$x^{-\lambda}(\Delta_g + 1)(x^\lambda v(y)) = I_{+1}(\lambda)v(y) + \mathcal{O}(x) \quad \text{as } x \to 0.$$Using the exact cylindrical model $-\partial_t^2 + \Delta_h + 1$, we explicitly compute this action as:$$I_{+1}(\lambda) = -\lambda^2 + \Delta_h + 1.$$The operator $\Delta_g + 1$ acts as a Fredholm operator $x^\alpha H^2_b(M) \to x^\alpha L^2_b(M)$ if and only if the weight $\alpha$ does not intersect the real parts of any roots of the indicial operator. We must analyze the indicial roots satisfying $I_{+1}(\lambda)v = 0$. Letting $v(y) = \phi_j(y)$ be the orthonormal eigenfunctions of the cross-sectional Laplacian with eigenvalues $\mu_j^2 \ge 0$, the roots are exactly given by:$$\lambda = \pm\sqrt{\mu_j^2 + 1}.$$

Because of the $+1$ shift, the roots are located at $\pm 1$. There is no double zero. Instead, the positive and negative indicial roots are strictly separated by a spectral gap of width $2$ centered on the imaginary axis. We can simply draw a line between roots to define our integration contour since our indicial roots are completely separated. Any weight $\alpha$ chosen strictly between the positive and negative roots $$-1 < \alpha < 1$$avoids the poles entirely. By placing our contour within this pole-free strip (most naturally on the imaginary axis, which corresponds to the unweighted space where $\alpha = 0$), the Mellin integral converges absolutely. This clean separation guarantees that the resulting parametrix maps the standard, unweighted space $L^2(M) \to H^2_b(M)$ as a bounded Fredholm operator. The resulting kernel is conormal to the diagonal smoothly up to the front face, and vanishes to first order at both the left and right boundaries.

Since the cross-sectional eigenvalues are nonnegative, the shifted operator $\Delta_h + 1$ has strictly positive eigenvalues. Consequently, there is no double root at zero. The indicial roots are strictly bounded away from the imaginary axis, with the roots located at $\pm 1$. Because our indicial roots are completely separated by a spectral gap, we can safely draw our contour line (which corresponds to choosing the weight $\alpha$) anywhere within this gap. In particular, the unweighted space $L^2(M)$, which corresponds to $\alpha = 0$, falls into the permissible strip $-1 < \alpha < 1$. Thus, unlike the pure Laplacian, our operator $\Delta_g + 1$ is Fredholm on the standard, unweighted $b$-Sobolev spaces.

\subsection{Asymptotic Expansions}
\subsubsection{Expansion Index}

The expansion index sets $\mathcal{E}_{\text{lb}}$ and $\mathcal{E}_{\text{rb}}$ dictate the asymptotic behavior of the parametrix near the left and right boundaries of the blown-up space. For a general $b$-operator, an expansion index is formally defined as a two-dimensional subset, typically residing in $\mathbb{C} \times \mathbb{N}_0$. This two-dimensional structure is required to track both the polynomial exponent $\lambda$ and the integer power $k$ of any associated logarithmic terms $(\log s)^k$, which inevitably arise from the indicial roots (such as a double zero).%%inevitably arise from the collision of indicial roots (such as a double zero). 

\noindent
 In the definition $\mathbb{C} \times \mathbb{N}_0$, the second coordinate tracks the integer power $k$ of the logarithmic term $(\log s)^k$. A simple pole generates a term with no logarithms, which corresponds exactly to $(\log s)^0$. If we restrict this definition to $\mathbb{N}$, we formally assert that $k \ge 1$. This implies every single term in the asymptotic expansion must contain at least a $\log s$, which breaks the general definition and completely invalidates our specific case where $k$ is uniformly zero.

Because of these potential logs, the most general parametrix expansion is structurally complex. However, we can express a simpler expansion devoid of logarithmic terms, as presented in Proposition 3.1 of \cite{DyLi24}:
\[
    \mathbf{Q}(x, s) \sim \sum_{\lambda \in \mathcal{E}_{\text{rb}}} s^\lambda v_s(x)
\]
It is crucial to specify that this represents a special, simplified case and is not what the most general parametrix looks like. 

Our shifted operator $\Delta_g + 1$ perfectly aligns with this simplest case. Because our indicial roots are distinct simple poles separated by a spectral gap, there are no double roots. Consequently, the native parametrix for our operator generates no logarithmic terms, and the logarithmic power $k$ is uniformly zero.

By placing the Mellin contour directly on the imaginary axis, the contour cleanly splits the positive and negative indicial roots. Because there are no logs, our expansion index sets effectively reduce to the one-dimensional discrete sets of these shifted roots:
\[
    \mathcal{E}_{\text{rb}} = \left\{ \sqrt{\mu_j^2 + 1} \;\middle|\; j \in \mathbb{N}_0 \right\}
\]
\[
    \mathcal{E}_{\text{lb}} = \left\{ \sqrt{\mu_j^2 + 1} \;\middle|\; j \in \mathbb{N}_0 \right\}
\]
The spectrum of the positive Laplacian $\Delta_B$ on a compact cross-section always includes the constant mode, meaning the lowest eigenvalue is $\mu_0^2 = 0$. If we use $\mathbb{N}$ (starting at 1) for the eigenvalue index $j$, we accidentally omit this bottom of the spectrum, which corresponds to the crucial minimal shifted root $\sqrt{0 + 1} = 1$. To include this root, the index must go as $j = 0, 1, 2, \dots$.

\subsubsection{The Leading Order / Minimum $E_{rb}, E_{lb}$}
The $E_{\text{lb}}$ and $E_{\text{rb}}$ are simply the minimums of those index sets:$$E_{\text{rb}} = \min \mathcal{E}_{\text{rb}} = \sqrt{\mu_0^2 + 1} = 1$$While the index set $\mathcal{E}$ tells us every term in the infinite Taylor-like series, the minimum $E$ is just the very first term. It dictates the leading order of vanishing, which is all we actually need to determine the Sobolev space mapping properties (the $\alpha$ and $\beta$ bounds).

\subsubsection{Cutting the Indicial Roots}

Let's see exactly how the contour cuts the indicial roots to form $\mathcal{E}_{\text{rb}}$ and $\mathcal{E}_{\text{lb}}$, and explains the logarithmic terms when solving $(\Delta_g + 1)u = f$.
To fully characterize the parametrix $\mathbf{Q}$, we must define its polyhomogeneous expansions at the left and right boundaries of the blown-up double space. The index sets $\mathcal{E}_{\text{lb}}$ and $\mathcal{E}_{\text{rb}}$ governing these expansions are determined precisely by the indicial roots and our choice of integration contour.

%by the indicial roots and our choice of integration contour in the Mellin plane.

% vertical contour $\Re(\lambda) = \alpha$ in the complex plane
Recall that the indicial roots of $\Delta_g + 1$ are perfectly separated:$$\lambda_j^\pm = \pm \sqrt{\mu_j^2 + 1}$$When we invert the indicial operator on the front face, we draw a contour in the complex plane to evaluate the inverse Mellin transform. As demonstrated in Proposition 3.1, the asymptotic expansion at the right boundary is determined by closing the contour in one half-plane, thereby capturing the poles on that side. The expansion at the left boundary  is determined by closing the contour in the opposite half-plane. By choosing our contour on the imaginary axis ($\alpha = 0$), we cleanly cut the roots into two disjoint sets.
\begin{itemize}
    \item [] The right boundary: The expansion here is dictated by the positive roots residing to the right of the contour. Thus, the base index set is generated by $\lambda_j^+ = +\sqrt{\mu_j^2 + 1}$.
    \item [] The left boundary: The expansion here is dictated by the negative roots to the left of the contour. Because the boundary defining function $s'$ is inverted relative to $s$, these negative roots translate into positive powers, generating an index set corresponding to $\vert{}\lambda_j^-\vert{} = +\sqrt{\mu_j^2 + 1}$.
\end{itemize}
Let $E_{\text{lb}} = \min \mathcal{E}_{\text{lb}}$ and $E_{\text{rb}} = \min \mathcal{E}_{\text{rb}}$ denote the leading order of vanishing at these boundaries. Because the cross-sectional eigenvalues satisfy $\mu_j^2 \ge 0$, the minimal roots are strictly bounded away from zero:$$E_{\text{lb}} = 1, \quad E_{\text{rb}} = 1$$

\subsubsection{Mapping Properties from the Index Sets}
%This geometric data directly verifies our allowable
Let's verify our allowable Sobolev weights. According to Proposition 3.1 in \cite{DyLi24}, the operator maps $x^\alpha H^s_b \to x^\beta H^{s+N}_b$ provided $\beta \le \alpha$, $\alpha > -E_{\text{rb}}$, and $\beta < E_{\text{lb}}$. Substituting our leading orders, we require:$$\alpha > -1 \quad \text{and} \quad \beta < 1$$This recovers the exact stability strip $-1 < \alpha < 1$ required for the unweighted $L^2 \to H^2_b$ Fredholm property. Furthermore, when analyzing the parametrix remainder $\mathbf{R}$, Proposition 3.2  in \cite{DyLi24} guarantees that if the kernel vanishes to infinite order at the front face, the restriction $\beta \le \alpha$ is lifted, allowing the remainder to strictly improve the decay rate.

We defined $E_{\text{lb}}$ and $E_{\text{rb}}$ as positive minimums, which aligns with formula in \cite{DyLi24} that $\alpha > -E_{\text{rb}}$ and $\beta < E_{\text{lb}}$.

\subsection{Full Parametrix Asymptotics}

For the resolvent family and spectral poles, \cite{DyLi24} setup, the foundational component for Region I is the meromorphic family of operators $(\sigma^2 + \Delta_y)^{-1}$. For our project, because we are studying the shifted operator $\Delta_g + 1$, our resolvent family on the cylindrical end is given by:$$(\sigma^2 + \Delta_y + 1)^{-1}$$or equivalently in terms of the Mellin parameter $\lambda = i\sigma$:$$(-\lambda^2 + \Delta_y + 1)^{-1}$$We define the corresponding $b$-Spectrum for our operator as the set of points where the inverse fails to exist:$$\text{Spec}_{\text{b}}(\Delta_g + 1) = \{ \sigma \in \mathbb{C} \mid \sigma^2 + \Delta_y + 1 \text{ is not invertible} \}$$

\noindent As we see in \cite{Mel93}, the set of $\text{Spec}_{\text{b}}(\Delta_g+1)$ is the indicial roots of $\Delta_g+1$. This convention is designed so that the indicial family is simply given by the Fourier transform of the Schwartz kernel restricted to front face using the $\tau$ coordinates.

The poles of our resolvent family $(\sigma^2 + \Delta_y + 1)^{-1}$ occur precisely when $\sigma^2 = \mu_j^2 + 1$, yielding the poles at:$$\sigma = \pm i\sqrt{\mu_j^2 + 1}$$

We should also discuss the absence of the double pole at $\sigma = 0$, in pure Laplacian case, the cross-sectional eigenvalue $\mu_0 = 0$ creates a pole at $\sigma = 0$ of order 2, which requires the Laurent expansion:\[
(\sigma^2 + \Delta_y)^{-1} = \frac{\Pi_{-z^2}}{\sigma^2 - z^2} + R_{\text{hol}}(\sigma)
\label{Laurent expansion} \tag{Laurent expansion}
\]
containing the double pole responsible for generating logarithmic terms. 

\vspace{1em}
\noindent
For our shifted operator $\Delta_g + 1$, the situation is simpler since the cross-sectional eigenvalues satisfy $\mu_j^2 \ge 0$, our lowest squared eigenvalue is $\mu_0^2 + 1 = 1$. Therefore, $\sigma = 0$ is never a pole. All poles are simple: Every pole in our resolvent family is strictly located at $\sigma = \pm i\sqrt{\mu_j^2 + 1}$ and has order 1. There are no higher-order poles anywhere in the spectrum.

\vspace{1em}
\noindent\textbf{Remark.} If the poles are simple, then the term $(\log s)^k$ in equation (3.22) in \cite{DyLi24} disappears. In that case, the only nonsimple pole is the pole at zero, which has order 2. We only pick up this log term at zero. We don't pick up log term at other places, we make it simple.
\noindent It will be a lot more easier when we see the exact formula, which is \eqref{Laurent expansion}. So if we apply the residue theorem to this, we see what exactly these $A_{z,k}$ are. They are just projections on eigenspaces.\\

\subsubsection{Behavior of the Parametrix $\mathbf{Q}$ Near the Right Boundary and Left Boundary}
To analyze the behavior of the parametrix $\mathbf{Q}$ near the right boundary for the shifted operator $\Delta_g + 1$, we evaluate the resolvent integral in the coordinates $s = e^\tau$ as $s \to 0$. For the integral representation, we use Region II coordinates where $s < 1$, we mirror Dyatlov's equation (3.20). The parametrix is given by the inverse Fourier transform of the shifted resolvent:$$\mathbf{Q}(s) = \frac{1}{2\pi} \int_{\mathbb{R}-ia} s^{i\sigma} (\sigma^2 + \Delta_B + 1)^{-1} d\sigma$$By integrating by parts, this operator is smoothing on the compact cross-section $B$ away from the diagonal and satisfies the bound $\Vert{}\mathbf{Q}(s)\Vert{}_{H^\ell(B) \to H^{\ell+N}(B)} \lesssim s^a$ as $s \to 0$.

Let's see the contour deformation and general expansion. To extract the precise asymptotics towards right boundary, we deform the contour down to $\mathbb{R} - ia'$, crossing the poles of the resolvent in the lower half-plane. Adapting Dyatlov's equation (3.22) directly to our spectrum yields:$$\mathbf{Q}(s) = \sum_{\substack{z \in \text{Spec}_{\text{b}}(\Delta_g+1) \\ -a' < \text{Im } z < -a \\ k \le \text{ord}(z) - 1}} s^{iz}(\log s)^k A_{z,k} + \int_{\mathbb{R}-ia'} s^{i\sigma} (\sigma^2 + \Delta_B + 1)^{-1} d\sigma$$

\noindent
As for the absence of the logarithmic terms, our shifted operator drastically simplifies the resulting expansion structure:
The poles in the lower half-plane are precisely located at $z_j = -i\sqrt{\mu_j^2 + 1}$.
 Because the eigenvalues satisfy $\mu_j^2 \ge 0$, the closest pole to the real axis is safely distanced at $\text{Im } z_0 = -1$.
The value of $z_j$ is never zero. Consequently, every single pole in our resolvent is strictly simple, meaning $\text{ord}(z_j) = 1$ universally.
Because the upper bound for the logarithmic power is $k \le \text{ord}(z_j) - 1$, the index $k$ is forced to be exactly $0$ for all terms.

Let's see the final asymptotics. Near the right boundary, substituting our simple poles $z_j$ into the exponent $s^{iz}$ yields $s^{i(-i\sqrt{\mu_j^2 + 1})} = s^{\sqrt{\mu_j^2 + 1}}$. The expansion collapses into a clean, log-free sum of pure polynomial powers:$$\mathbf{Q}(s) = \sum_{a < \sqrt{\mu_j^2 + 1} < a'} s^{\sqrt{\mu_j^2 + 1}} A_j + \mathbf{R}(s)$$The remainder integral $\mathbf{R}(s)$ decays faster than $s^{a'}$ as $s \to 0$. The coefficients $A_j$ remain smoothing operators, defined as multiples of the orthogonal projections onto the eigenspaces of $\Delta_B$.

The part of  left boundary will be quite similar as following.
While the complex contour integration differs in direction, the final asymptotic behavior of the parametrix $\mathbf{Q}$ near both boundaries is perfectly symmetric. The coordinate becomes $s' = e^{-\tau} \to 0$. The contour deformation is now the upper half-plane with the poles $+i\sqrt{\mu_j^2 + 1}$. Exponent $(s')^{-iz_j}$ and we have the final asymptotic power $(s')^{\sqrt{\mu_j^2 + 1}}$.

The complex analysis is perfectly mirrored. We deform the contour down to capture the negative imaginary poles for right boundary, and we deform it up to capture the positive imaginary poles for left boundary.

The negative sign in the left boundary exponent formula $(s')^{-iz_j}$ perfectly cancels the positive sign of the upper half-plane poles.
Because of this cancelation, the final asymptotic behavior is completely identical at both boundaries. The parametrix $\mathbf{Q}$ decays at the exact same polynomial rate as we approach either the left or right boundary of the blown-up space.

\begin{proposition} The real axis contains no indicial roots, so $(\mathbb{R} + i0) \cap \text{Spec}_{\text{b}}(\Delta_g + 1) = \emptyset$. Setting $a=0$, the orders of vanishing are explicitly bounded by the spectral gap:
\[
    E_{\text{rb}} := -\max\{\text{Im } z : z \in \text{Spec}_{\text{b}}(\Delta_g + 1), \text{Im } z < 0\} = 1,
\]
\[
    E_{\text{lb}} := \min\{\text{Im } z : z \in \text{Spec}_{\text{b}}(\Delta_g + 1), \text{Im } z > 0\} = 1.
\]
Fix $\chi \in C^\infty_c([0,1))$ be such that $\chi(r) = 1$ near $r = 0$ and $\chi(r)$ as a function on $M$ is supported in the exact cylindrical end. Define the cutoff parametrix
\[
    \widetilde{\mathbf{Q}}(\tau, r, y, y') := \frac{1}{2\pi} \chi(r) \int_{\mathbb{R}} e^{i\sigma\tau}(\sigma^2 + \Delta_B + 1)^{-1} d\sigma.
\]
Then
\[
    (\Delta_g + 1)\widetilde{\mathbf{Q}}(\tau, r, y, y') = \delta(\tau)\delta(y - y') \quad \text{in a small neighborhood of ff},
\]
and
\[
    \widetilde{Q} : x^\alpha H^\ell_{\text{b}}(M) \to x^\beta H^{\ell+2}_{\text{b}}(M)
\]
is bounded for all $\beta \le \alpha$, $\alpha > -1$, and $\beta < 1$, where $\widetilde{Q}$ is the operator with kernel $\widetilde{\mathbf{Q}}$.

\end{proposition}

\vspace{1em}
Let's see full parametrix asymptotics in region I, as the integration by parts we see in \cite{DyLi24} (such as Equation 3.17), the Schwartz kernel $\mathbf{Q}(\tau, y, y')$ in Region I is evaluated via contour integration over $\sigma$:$$\mathbf{Q}(\tau, y, y') = \frac{1}{2\pi} \int_{\mathbb{R} - i\alpha} e^{i\sigma \tau} (\sigma^2 + \Delta_y + 1)^{-1} \, d\sigma$$Because all poles of our resolvent are simple, applying Cauchy's residue theorem as $\tau \to \infty$ does not produce any logarithmic terms induced by derivative (unlike the $\log s$ factors that arise from double poles). Instead, the full parametrix asymptotics simplify to a pure polyhomogeneous expansion consisting exclusively of exponential decay terms:$$\mathbf{Q}(\tau, y, y') \sim \sum_{j=0}^{\infty} e^{-\vert{}\tau\vert{}\sqrt{\mu_j^2 + 1}} \Pi_j(y, y')$$where $\Pi_j(y, y')$ are the smooth, finite-rank orthogonal projections onto the eigenspaces of $\Delta_y$ corresponding to the eigenvalue $\mu_j^2$.

In \cite{DyLi24}, the full asymptotics require tracking logarithmic residues due to the zero mode. For our $\Delta_g + 1$ project, the $+1$ spectral gap ensures that the full parametrix asymptotics are entirely free of logarithmic terms, yielding a clean, classical exponential decay expansion governed exclusively by the shifted indicial roots $\sqrt{\mu_j^2 + 1}$.

\newpage

\section{Mapping Properties}

\noindent
For the mapping property here, we can see two theorems from \cite{Mel93}. The point is to know more about the boundedness properties of the parametrix on the $b$-Sobolev spaces.

\subsection{Mapping Properties of $b$-pseudodifferential Operators}
The following theorem is Theorem 5.34 from \cite{Mel93}.

\newtheorem{theorem}{Theorem}[section]
\begin{theorem} \ Any $A \in \widetilde{\Psi}_{\text{b}}^{m,\mathcal{E}}(X; {}^b\Omega^{\frac{1}{2}})$, $m \in \mathbb{Z}$, is bounded as an operator
\begin{equation*}
\begin{gathered}
A: x^\alpha H_{b}^M(X; {}^b\Omega^{\frac{1}{2}}) \longrightarrow x^\beta H_{b}^{M-m}(X; {}^b\Omega^{\frac{1}{2}}), \ M \in \mathbb{Z}, \\
\text{provided } \beta \leq \alpha, \ \alpha + \inf E_{\mathrm{rb}} > 0, \ \inf E_{\mathrm{lb}} > \beta.
\end{gathered}
\end{equation*}

\end{theorem}

\begin{proof}
We can start with the decomposition of the operator.
    Decompose $A \in \widetilde{\Psi}_{\text{b}}^{m, \mathcal{E}}$ into a small calculus component $A_{\text{small}} \in \Psi_{\text{b}}^m$ and a residual component $R \in \widetilde{\Psi}_{\text{b}}^{-\infty, \mathcal{E}}$ of order $-\infty$.

We see that Melrose used the idea of small calculus boundedness $\Psi_{\text{b}}^m$.
(The small calculus of b pseudodifferential operators reduces to the ordinary calculus of pseudodifferential operators when the compact manifold has no boundary. In microlocal analysis, it is used to invert elliptic operators. See \cite{Mel93} for more details.)
Invariance under conjugation by complex powers $x^z$ of the defining function reduces the problem to unweighted spaces ($\alpha = \beta = 0$).
   Using Hörmander's symbolic construction, extract an approximate square root $-A^* A + C = B^* B - R$ with $B \in \Psi_{\text{b}}^0$.
 This reduces boundedness $A: H_{b}^M \to H_{b}^{M-m}$ to the case $M = m = 0$ and the remaining order $-\infty$ term. For the residual reduction ($m = -\infty$), we conjugate by powers of $x$ to reduce the general weight case to $\alpha = \beta = 0$, transforming the boundary constraints to $\inf E_{\text{rb}} > 0$ and $\inf E_{\text{lb}} > 0$.
The boundedness claim simplifies to establishing the quadratic form bound $|\langle Ru, u \rangle| \le C \|u\|_{L^2_{\text{b}}}^2$.
    
Next we do kernel localization and use weighted Cauchy--Schwarz estimate.
   Localize the kernel $\kappa(x, s, y, y')$ to a neighborhood of the front face and near a boundary face using cutoff functions and adjoint symmetry.
    Express the inner product in projective $b$-coordinates $(x, s = x'/x, y, y')$:
        \[
            \langle Ru, u \rangle = \int_\delta^\infty \int_0^\epsilon \int_{\partial X \times \partial X} u(x,y) \kappa(x, s, y, y') u(x/s, y') \frac{ds}{s} \frac{dx}{x} dy \, dy'.
        \]
      Insert split weights $s^{-r}$ and $s^r$ for a small parameter $r > 0$ such that $0 < r < \min\{\inf E_{\text{rb}}, \inf E_{\text{lb}}\}$, and apply the Cauchy--Schwarz inequality:
        \[
            |\langle Ru, u \rangle|^2 \le  \iiint |s^{-r} \kappa(x,s,y,y') u(x,y)|^2 \frac{ds}{s} \frac{dx}{x} dy \, dy' \times  \iiint |s^r u(x/s, y')|^2 \frac{ds}{s} \frac{dx}{x}dy \, dy' .
        \]
Integrating out $s$ yields a uniform bound by the square of $\|u\|_{L^2_{\text{b}}}$, completing the proof.

\end{proof}

\newpage

\noindent
The mapping property we are looking for is Theorem 5.34. This paper is basically presenting a special case for Theorem 5.34.

\vspace{1em}
\noindent
The next theorem is Theorem 5.40 from \cite{Mel93}.

\begin{theorem} If $P \in \operatorname{Diff}_{\text{b}}^k(X; {}^b\Omega^{\frac{1}{2}})$ is elliptic on a compact manifold with boundary $X$ then for each $\alpha \in \mathbb{R}$ and $M \in \mathbb{R}$ the null space
\begin{equation*}
\left\{u \in x^\alpha H_{b}^M(X; {}^b\Omega^{\frac{1}{2}}); Pu = 0\right\} \subset x^\alpha H_{b}^\infty(X; {}^b\Omega^{\frac{1}{2}})
\end{equation*}
\begin{center}
is finite dimensional
\end{center}
and
\begin{equation*}
\begin{aligned}
P: x^\alpha H_{b}^M(X; {}^b\Omega^{\frac{1}{2}}) &\longrightarrow x^\alpha H_{b}^{M-k}(X; {}^b\Omega^{\frac{1}{2}}) \text{ is Fredholm} \\
&\Longleftrightarrow \alpha \notin -\operatorname{Im}\operatorname{spec}_{\text{b}}(P).
\end{aligned}
\end{equation*}

\end{theorem}

\begin{proof}
The proof establishes that the $b$-differential operator $P$ is Fredholm if and only if $\alpha \notin -\operatorname{Im}\operatorname{spec}_{\text{b}}(P)$, and that its null space is finite-dimensional. 

First, we assume the sufficiency condition. Start by assuming the weight condition $\alpha \notin -\operatorname{Im}\operatorname{spec}_{\text{b}}(P)$ holds.
To construct parametrix, we utilize the existence of right parametrix and left parametrix, $E_R$ and $E_L$, such that $P \circ E_R = \operatorname{Id} - R_R$ and $E_L \circ P = \operatorname{Id} - R_L$. Then we can establish compactness of remainders. Identify that the remainder terms $R_R$ and $R_L$ are compact operators because they map into spaces with strictly improved regularity and weight ($x^{\alpha+\epsilon}H_{b}^\infty$).

Now we prove finite-dimensional null space. Observe that the null space of $P$ is contained within the null space of $\operatorname{Id} - R_L$. Because $R_L$ is compact, $\operatorname{Id} - R_L$ has a finite-dimensional null space, which forces the null space of $P$ to be finite-dimensional as well.
So we can generalize the null space property by extending this finite-dimensionality to arbitrary $\alpha$ by noting that the null space can only increase as the weight $\alpha$ decreases.

We also need to prove closed range and finite codimension, showing that the range of $P$ contains the range of $\operatorname{Id} - R_R$. Since $R_R$ is compact, the operator $\operatorname{Id} - R_R$ has a closed range of finite codimension. This implies $P$ also has a closed range of finite codimension, completing the proof that $P$ is Fredholm under the sufficiency condition.

Finally, we defer the necessity condition. The proof that the weight condition is necessary for $P$ to be Fredholm is deferred, as it is a direct consequence of the relative index formula proven later in Theorem 6.5.

\end{proof}

From those mapping properties, we can get the elliptic estimates.

\newpage

\subsection{Mapping Properties of Parametrix}

Applying Melrose's theorems directly to our shifted operator $\Delta_g + 1$ rigorously confirms the boundedness of our parametrix and guarantees the Fredholm property on the unweighted space. Because of the $+1$ spectral gap, our project fits the conditions of both theorems perfectly without requiring any complex contour shifting.

For Theorem 5.34 in \cite{Mel93}, this theorem establishes the precise mapping properties of our constructed parametrix $\mathbf{Q}$, which acts as the pseudodifferential operator $A$ in the theorem. The variable $x$ corresponds to our boundary defining function $x$. In our earlier analysis, the leading orders of vanishing at the left and right boundaries were strictly determined by the shifted root, meaning $\inf E_{\text{lb}} = 1$ and $\inf E_{\text{rb}} = 1$.

%lowest shifted root, meaning $\inf E_{\text{lb}} = 1$ and $\inf E_{\text{rb}} = 1$.

Substituting these values into the theorem's bounds yields the strict inequalities $\alpha > -1$ and $\beta < 1$. When mapping a space to itself with the same weight by setting $\beta = \alpha$, this confirms that our parametrix $\mathbf{Q}$ is a bounded operator on the weighted $b$-Sobolev space $x^\alpha H_b^M(M)$ for all weights safely within the open stability strip $\vert{}\alpha\vert{} < 1$.

We are using Theorem 5.40 in \cite{Mel93} to discuss the Fredholm property of $\Delta_g + 1$, this theorem provides the exact condition for our specific operator to be Fredholm. The theorem states $P$ is Fredholm if and only if the chosen weight $\alpha$ avoids the spectrum of the indicial family. In this notation, this means avoiding the specific shifted roots $\pm\sqrt{\mu_j^2 + 1}$.
Because the lowest cross-sectional eigenvalue is zero ($\mu_0^2 = 0$), the roots closest to zero are strictly located at $+1$ and $-1$. The weight $\alpha = 0$ sits squarely in the center of this gap. Therefore, the crucial condition $\alpha \notin -\text{Im spec}_b(P)$ is automatically satisfied for the unweighted space where $\alpha = 0$.

\begin{theorem}
For any weight $\alpha$ that is not in the set of indicial roots, the operator
\[
\Delta_g + 1 : x^\alpha H_b^{s+2}(M) \to x^\alpha H_b^s(M)
\]
is a Fredholm operator. Consequently, the space of solutions to the homogeneous equation
\[
(\Delta_g + 1)u = 0
\]
in $x^\alpha H_b^{s+2}(M)$ is finite-dimensional. 
\end{theorem}

\newpage

\section*{Acknowledgements}
\addcontentsline{toc}{section}{Acknowledgements}
I would like to thank my advisor Xuwen Zhu for introducing me to the field of microlocal analysis and for her constant encouragement. Her guidance was pivotal in shaping this research direction, and her insightful discussions were instrumental to the completion of this work. I would also like to thank Zhenhao Li for his inspiring suggestions and generous assistance during the preparation of this paper.

\section*{References}
\addcontentsline{toc}{section}{References}

\begingroup
\renewcommand{\refname}{}

\endgroup

\end{document}